\documentclass[a4paper,11pt]{article}

\usepackage{amsmath}
\usepackage{amssymb}
\usepackage{amsthm}
\usepackage{graphicx}
\usepackage{enumerate}
\usepackage[utf8]{inputenc}
\usepackage[english]{babel}
\usepackage{hyperref}
\newtheorem{theorem}{Theorem}[section]
\newtheorem{proposition}[theorem]{Proposition}
\newtheorem{lemma}[theorem]{Lemma}
\newtheorem{corollary}[theorem]{Corollary}

\theoremstyle{definition}
\newtheorem{definition}[theorem]{Definition}
\newtheorem{example}[theorem]{Example}

\begin{document}
	
\title{Boundedness of solutions of the first-order linear multidimensional difference equations}
\author{Andrii Chaikovskyi, Oleksandr Liubimov}
\date{March 9, 2025}
\maketitle

\begin{abstract}
We investigate the boundedness of solutions of the first order linear difference equation of the form $x_{n+1} = Ax_{n} + y_{n}, \; n \geq 1$ where $A$ is a square matrix with complex entries, sequence $\{y_{n}\}_{n\geq 1}$ and initial value $x_1$ are supposed to be known. Firstly, we discuss the one-dimensional case of this equation $x_{n+1} = ax_{n} + y_{n}, \; n \geq 1$ where $a$ is a complex number. In particular, we obtain the sufficient conditions for boundedness or unboundedness of the solutions in case $|a|=1$(the critical case) by considering the exponential sums of the forms $\sum y_{n}e(n\varphi)$ and $\sum e(f(n))$. 

Then we proceed to the investigation of the equation in the multidimensional case and reduce our problem to analysis of the spectrum and Jordan cells of matrix $A$. The problem is especially interesting when spectrum of $A$ contains eigenvalues $\lambda$ with $|\lambda|=1$. At the end of the article we obtain a theorem that reveals the connection between equations $x_{n+1} = ax_{n} + y_{n}, \; n \geq 1$ with $|a|=1$ and $x_{n+1} = Jx_{n} + y_{n}, \; n \geq 1$ with $J$ being a Jordan cell of an eigenvalue $\lambda$, $|\lambda|=1$.
\end{abstract}

% Please add you keywords in English 
% Будь ласка, додайте ключові слова англійською мовою
\noindent \textbf{Keywords:} Linear difference equation, bounded solution, exponential sums.

\noindent \textbf{AMS 2020 Subject Classification:} 39A06, 39A45, 11L03, 11L07.

\tableofcontents

\section{Introduction}
In space $\mathbb{C}^{m}$ with Euclidean norm we consider the following difference equation

\begin{equation}\label{general_difference_equation}
x_{n+1} = Ax_{n} + y_{n}, \; n \geq 1
\end{equation}

\noindent with respect to unknown sequence $\{x_{n} : n \geq 1\} \subset \mathbb{C}^{m}$. The first element of this sequence $x_{1}$, sequence $\{y_{n} : n \geq 1\} \subset \mathbb{C}^{m}$ and square matrix $A \in \mathcal{M}_{m}(\mathbb{C})$ of order $m$ are supposed to be known.
Various necessary and sufficient conditions for the existence and uniqueness of the bounded solution of this equation were previously investigated by different authors in finite-dimensional and infinite-dimensional cases (see \cite{HorLagoda1}, \cite{HorLagoda2}, \cite{ChaikDiff}, \cite{Kravets} and cited literature there). 

Difference equation \eqref{general_difference_equation} has the following well-studied cases:

1) If spectrum of $A$ is such that $\sigma(A)\ \subset\ \{ z\in\mathbb{C}\ |\ |z| < 1 \}$ then for any $x_1\in\mathbb{C}^m$ the solution is bounded;

2) If $\sigma(A)\ \subset\ \{ z\in\mathbb{C}\ |\ |z| > 1 \}$ then only for $x_{1} = - \sum_{k=1}^{\infty} A^{-k-1}y_{k}$ the solution is bounded;

3) If $\sigma(A)\ \cap\ \{ z\in\mathbb{C}\ |\ |z| = 1 \} = \varnothing$ then it is possible to decompose the space into cartesian product of two spaces, which satisfy the conditions of cases 1 and 2.

On the other hand, the case $\sigma(A)\ \cap\ \{ z\in\mathbb{C}\ |\ |z| = 1 \} \neq \varnothing$ always requires some additional investigation because the simple characterization for it is not known. Existence of the bounded solution of the difference equation \eqref{general_difference_equation} in this case depends on the input sequence $\{y_{n} : n \geq 1\}$ in a rather complicated way in comparison with the cases 1)-3). Our aim is to establish some interesting and significant instances of such dependency.
   
\section{One-dimensional equation}
We consider the difference equation of the form

\begin{equation}\label{one_dimensional_equation}
x_{n+1} = ax_{n} + y_{n}
\end{equation}

\noindent
with respect to unknown sequence $\{x_{n}\}_{n \geq 1}$. Numbers $x_{1}$, $a \in \mathbb{C}$ and the sequence $\{y_{n}\}_{n \geq 1} \subset \mathbb{C}$ are supposed to be known.

By induction it is not difficult to deduce that the solution of the equation \eqref{one_dimensional_equation} is

$$x_{n} = a^{n-1} x_{1} + \sum^{n-1}_{k=1} a^{n-k-1}y_{k}, \; \; n \geq 1.$$

For the critical case $|a| = 1$ we investigate the conditions for $\{y_{n}\}_{n \geq 1}$ and $x_{1}$ under which the solution $\{x_{n}\}_{n \geq 1}$ is bounded or unbounded.

We begin this section with the following observation.

\begin{proposition}\label{proposition_boundedness_criterion_one_dim_a_1}
    If $|a| = 1$, then the solution of the difference equation \eqref{one_dimensional_equation} is bounded $\Longleftrightarrow$ $\exists M > 0 \; \forall K \geq 0 \; : \; \left\vert \sum^{K}_{k=1} \frac{y_{k}}{a^{k}} \right\vert \leq M$.
\end{proposition}

\begin{proof}
The proof is a simple consequence from the solution formula above.
\end{proof}

The question of boundedness of solution $\{x_{n}\}_{n \geq 1}$ in case $|a| = 1$ is rather difficult. Unlike in other cases, the simple and full characterization is still not known. So particular results with sufficient conditions of boundedness are interesting.

We write $a = e^{-2\pi i \varphi}$, where $\varphi \in [0,1)$.
Therefore according to the Proposition \ref{proposition_boundedness_criterion_one_dim_a_1} solution $\{x_{n}\}_{n\geq1}$ of the difference equation \eqref{one_dimensional_equation} is bounded if and only if the sequence of sums

\begin{equation*}
    \left\{\sum^{N}_{n=1} y_{n}e^{2\pi i n\varphi} : N \geq 1\right\}
\end{equation*}

\noindent
is bounded.

Thus we see that our problem is equivalent to the study of boundedness of the exponential sums written above.

For real number $x$ define $e(x) := e^{2\pi i x}$.

In this subsection we are solely concerned with the boundedness of exponential sums 

\begin{equation}\label{exp_sums_unit_coef}
    \left\{\sum_{n = 1}^{N} e(f(n)) \; : \; N \geq 1 \right\}
\end{equation}

\noindent where $f: \mathbb{N} \to \mathbb{R}$ is some function.

\begin{definition}
    For the function $f: \mathbb{N} \to \mathbb{R}$ we define the \textit{finite difference of} $f$ to be

    $$\Delta f(n) := f(n+1)-f(n)$$

    \noindent and we define the \textit{second order finite difference of} $f$ to be

    $$\Delta^{2} f(n) := \Delta f(n+1)- \Delta f(n) = f(n+2) - 2f(n+1) + f(n)$$
\end{definition}

One of interesting results connected with the question of boundedness of exponential sums is following.

\begin{theorem}[\hypertarget{kusmin_landau_inequality}{Kusmin-Landau Inequality}]\label{kusmin_landau_inequality} 

Suppose function $f$ and real number $\theta$ are such that 

$$0 < \theta \leq \Delta f(1) \leq \Delta f(2) \leq \ldots \leq \Delta f(N-1) < 1-\theta$$

Then 

$$\left\vert \sum_{n=1}^{N} e\left(f(n)\right)\right\vert \leq \cot\left(\frac{\pi \theta}{2}\right)$$
    
\end{theorem}

\begin{proof}
    For the proof and detailed discussion see the paper of Mordell \cite{Mordell}.
\end{proof}

We would like to remark that if function $f$ has monotone derivative $f^{\prime}$ on segment $[1, N]$ and $\theta < f^{\prime}(x) < 1 - \theta$, then from Lagrange's mean value theorem it follows that $f$ satisfies the conditions of Theorem \ref{kusmin_landau_inequality}.

The proof of \hyperlink{kusmin_landau_inequality}{Kusmin-Landau Inequality} written in the book of Graham and Kolesnik \cite{Graham} presents us an interesting and useful identity

\begin{proposition}\label{identity_for_exp_sums}
    $$-\sum_{n = n_{0}}^{N} e(f(n)) = \frac{1}{2}\left(1 + i\cdot \cot(\pi \Delta f(N+1) )\right)\cdot e(f(N+1)) - \frac{1}{2}\left(1 + i\cdot \cot(\pi \Delta f(n_{0}) )\right)\cdot e(f(n_{0})) - $$

    $$ - \frac{1}{2}\sum_{n=n_{0}}^{N}e(f(n+1)) \cdot i\left(\cot(\pi \Delta f(n+1)) - \cot(\pi \Delta f(n))\right)$$
\end{proposition}

\begin{proof}
    We write 

    $$e(f(n)) = \frac{e(f(n+1)) - e(f(n))}{e(\Delta f(n)) - 1} = -\frac{1}{2}(e(f(n+1)) - e(f(n)))\cdot \left(1 + i\cdot \cot(\pi \Delta f(n) )\right)$$

    and using Abel's summation by parts formula we obtain the desired identity.
\end{proof}

\vspace{20pt}

Boundedness of exponential sums $\sum_{n\geq1}e(n\alpha), \; \alpha \in (0,1)$ and geometric intuition suggest us a natural question: 

\textit{If $\Delta f(n) \to \psi \notin \mathbb{Z}$ as $n \to \infty$, is it true that the sequence of exponential sums \eqref{exp_sums_unit_coef} is bounded?}

And the answer is Yes! if $\Delta f(n)$ converges fast enough. 
With the aid of Proposition \ref{identity_for_exp_sums} we prove the following Theorem.

\begin{theorem}\label{theorem_sufficient_condition_bounded_exp_sums}
    Let $f$ be such that $ \sum_{n\geq1}\vert \Delta^{2} f(n)\vert < \infty$. Then:
    
    \begin{enumerate}
        \item $\Delta f(n) \to \psi \in \mathbb{R}$ as $n \to \infty$.

        \item If $\psi \notin \mathbb{Z}$ then sequence of exponential sums

        $$\left\{\sum_{n = 1}^{N} e(f(n)) \; : \; N \geq 1 \right\}$$

\noindent        is bounded.
    \end{enumerate}
\end{theorem}

\begin{proof}

Firstly, since $\sum_{n\geq1}\vert \Delta^{2} f(n)\vert < \infty$ series $\sum_{n\geq1} \Delta^{2} f(n)$ is convergent. Therefore, because

$$\Delta f(N) = \Delta f(1) + \sum_{n=1}^{N-1} \Delta^{2} f(n)$$

\noindent we deduce that $\exists \lim_{n \to \infty} {\Delta f(n)} =: \psi \in \mathbb{R}$.

\vspace{10pt}

Now suppose that $\psi \notin \mathbb{Z}$.

There exist some $m \in \mathbb{Z}$, $n_{0} \in \mathbb{N}$ and $\theta \in (0, 1)$ such that for all $n \geq n_{0}$ we have

$$m + \theta \leq \Delta f(n) \leq m + 1 - \theta$$

It means that $|\sin (\pi \Delta f(n))| \geq \sin(\pi \theta) =: T$ for all $n\geq n_{0}$.

\vspace{10pt}

Thus for all $n \geq n_{0}$ we have

$$|\cot \pi (\Delta f(n+1)) - \cot(\Delta f(n))| = \frac{|\sin(\pi\Delta f(n+1) -  \pi\Delta f(n))|}{|\sin(\pi \Delta f(n+1))| |\sin(\pi \Delta f(n))|} = $$

$$ = \frac{|\sin(\pi\Delta^{2} f(n)|}{|\sin(\pi \Delta f(n+1))| |\sin(\pi \Delta f(n))|} \leq \frac{\pi|\Delta^{2} f(n)|}{T^2}$$

\vspace{10pt}

We conclude that for all $N \geq n_{0}$ we have:

$$\left\vert \sum_{n=n_{0}}^{N}e(f(n+1)) \cdot i\left(\cot(\pi \Delta f(n+1)) - \cot(\pi \Delta f(n))\right)\right\vert \leq$$

$$\leq \sum_{n=n_{0}}^{N} |\cot(\pi \Delta f(n+1)) - \cot(\pi \Delta f(n))| \leq \frac{\pi}{T^2} \sum_{n=n_{0}}^{N} |\Delta^{2} f(n)| \leq \frac{\pi}{T^2} \sum_{n\geq1}\vert \Delta^{2} f(n)\vert < \infty$$

Also

$$\left\vert\left(1 + i\cdot \cot(\pi \Delta f(N+1) )\right)\cdot e(f(N+1))\right\vert \leq 1 +|\cot(\pi \theta)|$$

Thus, according to Proposition \ref{identity_for_exp_sums}, sequence of exponential sums $\left\{\sum_{n = 1}^{N} e(f(n))\right\}_{N\geq 1}$ is bounded.
\end{proof}

With the help of this Theorem we can, for example, prove boundedness of following sequences of exponential sums:

\begin{example}

    \begin{enumerate}
        \item For all $\varphi \in \mathbb{R} \setminus \mathbb{Z}$
        $$\left\{ \sum_{n=1}^{N} e\left( n \varphi + \sqrt{n}\log(n)\right) : N \geq 1 \right\}$$

        \item For arbitrary $s > 1$ and $\varphi \in (0,1)$, $\varphi + \sum_{k \geq 2}\frac{\sin(k)}{k \log^{s}(k)} \notin \mathbb{Z}$ 

        $$\left\{ \sum_{n=2}^{N} e\left(n \varphi + \sum_{r=2}^{n}\sum_{k=2}^{r}\frac{\sin(k)}{k \log^{s}(k)}\right) : N \geq 2\right\}$$
        
    \end{enumerate}
    
\end{example}

\vspace{20pt}

We proceed to another observation, which is also inspired by geometric intuition.

\begin{theorem}\label{theorem_sufficient_condition_unbounded_exp_sums}
    If $\frac{1}{N} \sum_{n=1}^{N} |\Delta f(n)| \to 0, \; N \to \infty$ then sequence of exponential sums

    $$\left\{\sum_{n = 1}^{N} e(f(n)) \; : \; N \geq 1 \right\}$$

\noindent
    is unbounded.
\end{theorem}

\begin{proof}
    Put $S(N) := \sum_{n = 1}^{N} e(f(n))$ for $N \geq 1$.
    We suppose by contradiction that $\exists M > 0 \; \forall N \geq 1 \; : \; |S(N)| \leq M$.

    Using Abel's summation by parts formula we get:

    $$|S(N) - N| = \left\vert \sum_{n=1}^{N} (e(f(n)) - 1) \right\vert = 
    \left\vert \sum_{n=1}^{N} e(f(n))(1 - e(-f(n))) \right\vert = $$

    $$ = \left\vert (1 - e(-f(N+1)))S(N+1) - (1 - e(-f(1)))S(1) + \sum_{n=1}^{N} (e(-f(n+1)) - e(-f(n)))S(n+1) \right\vert \leq$$

    $$\leq 4M + \sum_{n=1}^{N}|e(-f(n+1)) - e(-f(n))||S(n+1)| \leq 4M + M \sum_{n=1}^{N} |-f(n+1) + f(n)| = $$

    $$ = 4M + M \sum_{n=1}^{N} |\Delta f(n)|$$

    There exists some $N_{0}$ such that for all $N \geq N_{0}$ :

    $$\sum_{n=1}^{N} |\Delta f(n)| \leq \frac{N}{2M}$$

\noindent    and therefore

    $$N - |S(N)| \leq |N - |S(N)|| \leq |S(N) - N| \leq 4M + \frac{1}{2}N$$

 \noindent   which implies that $|S(N)| \geq \frac{1}{2}N - 4M$.
    Contradiction.

\end{proof}

\begin{corollary}\label{corollary_sufficient_condition_unbounded_exp_sums}
    If $\Delta f(n) \to \psi \in \mathbb{Z}, \; n \to \infty$, then sequence of exponential sums

    $$\left\{\sum_{n = 1}^{N} e(f(n)) \; : \; N \geq 1 \right\}$$

\noindent
    is unbounded.
\end{corollary}

\begin{proof}
    Since $e(f(n)) = e(f(n) - n\psi)$ and $\Delta(f(n)-n\psi) = \Delta f(n) - \psi$ without loss of generality we may suppose that $\Delta f(n) \to 0, \, n \to \infty$.

    Hence, by Stolz-Cesàro Theorem, $\frac{1}{N} \sum_{n=1}^{N} |\Delta f(n)| \to 0, \; N \to \infty$. Corollary follows from the previous Theorem.
\end{proof}

Now we write $a = e^{-2\pi i\cdot \varphi} = e(-\varphi)$, where $\varphi \in [0,1)$, and study boundedness of the solutions of difference equation:

\begin{equation}\label{one_dimensional_equation_a_is_1}
    x_{n+1} = e(-\varphi)x_{n} + y_{n},\; n \geq 1.
\end{equation}

We would like to remark that Proposition \ref{proposition_boundedness_criterion_one_dim_a_1} implies that boundedness of solutions of \eqref{one_dimensional_equation_a_is_1} depends only on $\varphi$ and $\{y_{n}\}_{n\geq1}$ and does not depend on the initial value $x_{1}$.

\begin{proposition}\label{proposition_one_dim_bounded_solution_criterion_with_exp_sums}
    Every solution of the difference equation \eqref{one_dimensional_equation_a_is_1} is bounded if and only if following sequence of sums

    $$\left\{\sum_{n=1}^{N} y_{n} \cdot e(n\varphi) \, : \, N \geq 1 \right\}$$

\noindent
is bounded.
\end{proposition}

\begin{proof}
    Follows from the Proposition \ref{proposition_boundedness_criterion_one_dim_a_1}
\end{proof}

\begin{definition}
    By $\mathbb{W}_{\varphi}$ we denote the set of all sequences $\{y_{n}\}_{n\geq1} \subset \mathbb{C}$ for which every solution of difference equation \eqref{one_dimensional_equation_a_is_1} is bounded.
\end{definition}

We remark that $\mathbb{W}_{\varphi}$ is not empty because $\{0\}_{n\geq 1} \in \mathbb{W}_{\varphi}$. We make the following easy observation:

\begin{proposition}
    For every $\varphi \in [0,1)$ the following holds:

    \begin{itemize}
        \item If $\{z_{n}\}_{n\geq1}, \, \{w_{n}\}_{n\geq 1} \in \mathbb{W}_{\varphi}$ then $\{z_{n} + w_{n}\}_{n\geq 1} \in \mathbb{W}_{\varphi}$.

        \item If $\alpha \in \mathbb{C}$ and $\{z_{n}\}_{n\geq1} \in \mathbb{W}_{\varphi}$ then $\{\alpha z_{n}\}_{n \geq 1} \in \mathbb{W}_{\varphi}$.

        \item If $\{z_{n}\}_{n\geq1} \in \mathbb{W}_{\varphi}$, then $\{z_{n}\}_{n \geq 1}$ is a bounded sequence.

        \item If $\{b_{n}\}_{n\geq 1} \subset \mathbb{R}$ is monotone bounded sequence and $\{z_{n}\}_{n\geq1} \in \mathbb{W}_{\varphi}$, then $\{b_{n} z_{n}\}_{n\geq1} \in \mathbb{W}_{\varphi}$.
        
    \end{itemize}
\end{proposition}

\begin{corollary}
    For every $\varphi \in [0,1)$ $\mathbb{W}_{\varphi}$ is a vector space over $\mathbb{C}$. In addition, it is the vector subspace of the vector space of all bounded sequences in $\mathbb{C}$.
\end{corollary}

\begin{theorem}
 Let $f$ be such that $\sum_{n\geq1}\vert \Delta^{2} f(n)\vert < \infty$. Then $\Delta f(n) \to \psi \in \mathbb{R}$ as $n \to \infty$. If $\psi + \varphi \notin \mathbb{Z}$ and $y_{n} = e(f(n)), \, n \geq 1$ then every solution of \eqref{one_dimensional_equation_a_is_1} is bounded, that is, $\{e(f(n))\}_{n\geq 1} \in \mathbb{W}_{\varphi}$.
\end{theorem}

\begin{proof}
    Put $g(n) := f(n) + n\varphi$. Then $\Delta^{2} g(n) = \Delta^{2} f(n),\, n\geq1$ which means that $\sum_{n\geq1}\vert \Delta^{2} g(n)\vert < \infty$. In addition $\Delta g(n) \to \psi + \varphi,\, n\to \infty$. 

    $$\left\{\sum_{n=1}^{N} y_{n} \cdot e(n\varphi) \, : \, N \geq 1 \right\} = 
    \left\{\sum_{n=1}^{N} e(g(n)) \, : \, N \geq 1 \right\}$$

    By Theorem \ref{theorem_sufficient_condition_bounded_exp_sums} sequence of exponential sums $\left\{\sum_{n=1}^{N} e(g(n))\right\}_{N\geq1}$ is bounded. Hence by Proposition \ref{proposition_one_dim_bounded_solution_criterion_with_exp_sums} every solution of \eqref{one_dimensional_equation_a_is_1} is bounded.

\end{proof}

\begin{example}
    For all $\varphi \in (0, 1)$ every solution of the difference equation
    $$x_{n+1} = e(\varphi)x_{n} + e(n^{\frac{1}{2}} - n^{\frac{1}{3}} + \log(n)), \, n \geq 1$$
    is bounded.
\end{example}

\begin{theorem}
    Let $f$ be such that $\Delta f(n) \to \psi \in \mathbb{R}$ as $n \to \infty$. If $\frac{1}{N} \sum_{n=1}^{N} |\Delta f(n) + \varphi| \to 0, \; N \to \infty$ and $y_{n} = e(f(n)), \, n \geq 1$, then every solution of \eqref{one_dimensional_equation_a_is_1} is unbounded, that is, $\{e(f(n))\}_{n\geq1} \notin \mathbb{W}_{\varphi}$.
\end{theorem}

\begin{proof}
    Put $g(n) := f(n) + n\varphi$. Then $\Delta g(n) = \Delta f(n) + \varphi, \, n \geq 1$. So $\frac{1}{N} \sum_{n=1}^{N} |\Delta g(n)| \to 0, \; N \to \infty$. Therefore, by Theorem \ref{theorem_sufficient_condition_unbounded_exp_sums} sequence of exponential sums $\left\{\sum_{n=1}^{N} e(g(n))\right\}_{N\geq1}$ is unbounded.
    Because $\left\{\sum_{n=1}^{N} y_{n} \cdot e(n\varphi)\right\}_{N\geq 1} = 
    \left\{\sum_{n=1}^{N} e(g(n))\right\}_{N \geq 1}$ Proposition \ref{proposition_one_dim_bounded_solution_criterion_with_exp_sums} implies that every solution of the difference equation \eqref{one_dimensional_equation_a_is_1} is unbounded.
\end{proof}

\begin{corollary}
    Let $f$ be such that $\Delta f(n) \to \psi \in \mathbb{R}$ as $n \to \infty$. If $\psi + \varphi \in \mathbb{Z}$ and $y_{n} = e(f(n)), \, n \geq 1$, then every solution of \eqref{one_dimensional_equation_a_is_1} is unbounded, that is, $\{e(f(n))\}_{n\geq1} \notin \mathbb{W}_{\varphi}$.
\end{corollary}

\begin{proof}
    Put $g(n) := f(n) + n\varphi$. Then $\Delta g(n) = \Delta f(n) + \varphi, \, n \geq 1$. So $\Delta g(n)$ converges to an integer. Therefore, by Corollary \ref{corollary_sufficient_condition_unbounded_exp_sums} sequence of exponential sums $\left\{\sum_{n=1}^{N} e(g(n))\right\}_{N\geq1}$ is unbounded.
    Because $\left\{\sum_{n=1}^{N} y_{n} \cdot e(n\varphi)\right\}_{N\geq 1} = 
    \left\{\sum_{n=1}^{N} e(g(n))\right\}_{N \geq 1}$ Proposition \ref{proposition_one_dim_bounded_solution_criterion_with_exp_sums} implies that every solution of the difference equation \eqref{one_dimensional_equation_a_is_1} is unbounded.
\end{proof}

\begin{example}
    For all $\alpha \in (0,1)$ all the solutions of the following difference equation are unbounded 

    $$x_{n+1} = x_{n} + e(n^{\alpha}), \, n \geq 1$$
\end{example}

\begin{example}
    All the solutions of the difference equation
    $$x_{n+1} = e^{i\frac{\pi}{2}}x_{n} + e^{i \sum_{k=1}^{n} \arctan(k)}, \, n \geq 1$$

\noindent    are unbounded.
\end{example}

\section{Multidimensional equation}

In space $\mathbb{C}^{M}$ with Euclidean norm we consider following difference equation
\begin{equation}\label{multidimensional_equation}
    x(n+1) = Ax(n) + y(n), \; n \geq 1
\end{equation}

\noindent 
with respect to unknown sequence $\{x(n) : n \geq 1\} \subset \mathbb{C}^{M}$. The first element of this sequence $x(1)$, sequence \\$\{y(n) : n \geq 1\} \subset \mathbb{C}^{M}$ and $M \times M$ square matrix $A$ with complex entries are supposed to be known.

Our goal is to find sufficient conditions for the sequence $\{y(n) : n \geq 1\} \subset \mathbb{C}^{M}$ and value $x(1)$ which guarantee boundedness of the solution in the critical case $\sigma(A)\cap \{ z\in\mathbb{C}\ |\ |z| = 1 \} \neq \varnothing.$

\vskip 5pt
\textbf{Remark.}
We will use Hilbert-Shmidt norm of matrix $A = (a_{ij})\in \mathcal{M}_{M}(\mathbb{C})$ which is defined as $\Vert A \Vert_2 = \sqrt{\sum_{i, j} |a_{ij}|^{2}}$.

Note that Cauchy-Bunyakovskyi-Schwartz inequality implies 
$$\forall x \in \mathbb{C}^{M} \; \forall A \in \mathcal{M}_{M}(\mathbb{C}) \; : \Vert Ax \Vert \leq \Vert A \Vert_2 \Vert x \Vert.$$

\vskip 10pt
Now we introduce the following notation:

$$(x)^{[n]} := \prod_{i=0}^{n-1}(x+i),\qquad (x)_{[n]} := \prod_{i=0}^{n-1}(x-i)$$

These polynomials are widely known as \textit{factorial polynomials} or, more precisely, \textit{raising} and \textit{falling factorials} respectively. They satisfy the useful identities listed below.

\begin{proposition}
	$\binom{x}{n} = \frac{(x)_{[n]}}{n!}$, $\binom{x+n-1}{n} = \frac{(x)^{[n]}}{n!}$ for all $x \in \mathbb{R}$ and $n \in \mathbb{N}$.
\end{proposition}

\begin{proposition}
	$(-x)_{[n]} = (-1)^{n} \cdot (-x)^{[n]}$
\end{proposition}

\begin{proposition}
	$$(x+y)_{[n]} = \sum_{i=0}^{n} \binom{n}{i} (x)_{[i]} \cdot (y)_{[n-i]}$$
\end{proposition}

\begin{proof}
	Induction on $n$.
\end{proof}

These properties immediately imply the following lemma, which will be of great importance for us later.

\begin{lemma}\label{lemma_identity_binomial}
	$$\binom{n-k}{m} = \sum_{j=0}^{m} (-1)^{m-j} \binom{n}{j} \frac{(k)^{[m-j]}}{(m-j)!}$$
\end{lemma}

\vskip 10pt
We proceed to the analysis of the multidimensional equation.  We remark that the computations below are well known, however, for completeness, we would like to include them in our article.

Let $J$ be a Jordan Normal Form (JNF) of the matrix A. There exists an invertible matrix $T$ such that $A = T J T^{-1} \; \Longleftrightarrow \; J = T^{-1} A T$. Multiply the equation \eqref{multidimensional_equation} at left by matrix $T^{-1}$ and obtain:

$$T^{-1} x(n+1) = T^{-1} Ax(n) + T^{-1}y(n) \; \Longleftrightarrow$$

$$\Longleftrightarrow \; T^{-1} x(n+1) = T^{-1} A T \left(T^{-1} x(n)\right) + T^{-1}y(n) \; \Longleftrightarrow$$

$$\Longleftrightarrow \; T^{-1} x(n+1) = J \left(T^{-1} x(n)\right) + T^{-1}y(n)$$

Put $\{x^{\prime}(n): n \geq 1\} := \{T^{-1} x(n)\}$, $\{y^{\prime}(n): n \geq 1\} := \{T^{-1} y(n)\}$ and obtain the difference equation

\begin{equation}\label{multidimensional_equation_jordan_normal_form}
    x^{\prime}(n+1) = J x^{\prime}(n) + y^{\prime}(n), \; n \geq 1
\end{equation}

Nonsingularity of the matrix $T$ implies that the solution of the difference equation \eqref{multidimensional_equation} is bounded if and only if the solution of the difference equation \eqref{multidimensional_equation_jordan_normal_form} is bounded, thus, henceforth we will investigate the equation \eqref{multidimensional_equation_jordan_normal_form}.

%%%%%%%%%%%%%%%%%%%%%%%%%%%%

In this way our problem is reduced to the study of the difference equation of form:

\begin{equation}\label{multidimensional_equation_jordan_cell}
    x(n+1) = J_{\lambda} x(n) + y(n), \; n \geq 1
\end{equation}

\noindent
in which $J_{\lambda} = \begin{pmatrix}
    \lambda & 1 & 0 & \ldots & 0 \\
    0 & \lambda & 1 & \ldots & 0 \\
    \vdots & \vdots & \vdots & \ddots & \vdots \\
    0 & 0 & 0 & \lambda & 1 \\
    0 & 0 & 0 & 0 & \lambda
\end{pmatrix}$ is Jordan cell which corresponds to the eigenvalue $\lambda$ of the matrix $A$ from equation \eqref{multidimensional_equation}.

If order of $J_{\lambda}$ is 1, then equation \eqref{multidimensional_equation_jordan_cell} is the First Order Linear Difference Equation in $\mathbb{C}$. 

\hspace{1pt}

If order of $J_{\lambda}$ is $M \geq 2$, then we do the following.
 
Let $x(n) = \begin{pmatrix}
    x_{1}(n) \\
    x_{2}(n) \\
    \vdots \\
    x_{M}(n)
\end{pmatrix}$, $y(n) = \begin{pmatrix}
    y_{1}(n) \\
    y_{2}(n) \\
    \vdots \\
    y_{M}(n)
    \end{pmatrix}$
 then equation \eqref{multidimensional_equation_jordan_cell} can be rewritten as:

$$\begin{pmatrix}
    x_{1}(n+1) \\
    x_{2}(n+1) \\
    \vdots \\
    x_{M-1}(n+1) \\
    x_{M}(n+1)
\end{pmatrix} = \begin{pmatrix}
    \lambda & 1 & 0 & \ldots & 0 \\
    0 & \lambda & 1 & \ldots & 0 \\
    \vdots & \vdots & \vdots & \ddots & \vdots \\
    0 & 0 & 0 & \lambda & 1 \\
    0 & 0 & 0 & 0 & \lambda
\end{pmatrix} \begin{pmatrix}
    x_{1}(n) \\
    x_{2}(n) \\
    \vdots \\
    x_{M-1}(n) \\
    x_{M}(n)
\end{pmatrix} + \begin{pmatrix}
    y_{1}(n) \\
    y_{2}(n)\\
    \vdots \\
    y_{M-1}(n) \\
    y_{M}(n)
\end{pmatrix} \; \Longleftrightarrow$$

$$\Longleftrightarrow \; \begin{pmatrix}
    x_{1}(n+1) \\
    x_{2}(n+1) \\
    \vdots \\
    x_{M-1}(n+1) \\
    x_{M}(n+1)
\end{pmatrix} = \begin{pmatrix}
    \lambda x_{1}(n) + x_{2}(n) + y_{1}(n) \\
    \lambda x_{2}(n) + x_{3}(n) + y_{2}(n) \\
    \vdots \\
    \lambda x_{M-1}(n) + x_{M}(n) + y_{M-1}(n) \\
    \lambda x_{M}(n) + y_{M}(n)
\end{pmatrix}, \; n \geq 1$$

\hspace{1pt}

\hspace{1pt}

Solution of difference equation $x_{M}(n+1) = \lambda x_{M}(n) + y_{M}(n)$ is $$x_{M}(n) = \lambda^{n-1}x_{M}(1) + \sum_{k=1}^{n-1}\lambda^{n-k-1}y_{M}(k).$$

Then the solution of difference equation $x_{M-1}(n+1) = \lambda x_{M-1}(n) + x_{M}(n) + y_{M-1}(n)$ is

$$x_{M-1}(n) =  \lambda^{n-1}x_{M-1}(1) + \sum_{k=1}^{n-1}\lambda^{n-k-1}\left(x_{M}(k) + y_{M-1}(k)\right).$$

Hence, solution of equation $x_{i}(n+1) = \lambda x_{i}(n) + x_{i+1}(n) + y_{i}(n)$ is

$$x_{i}(n) =  \lambda^{n}x_{i}(1) + \sum_{k=1}^{n-1}\lambda^{n-k-1}\left(x_{i+1}(k) + y_{i}(k)\right).$$

Thus we found the solution of difference equation \eqref{multidimensional_equation_jordan_cell}.

Now we turn our attention to the investigation of the case  $|\lambda| = 1$ and our idea is to find the connection between one-dimensional and multidimensional equations in this case. At the end of this section we will describe an interesting statement which, in fact, reveals this connection.

By induction it is not difficult to deduce that the solution of \eqref{multidimensional_equation_jordan_cell} can be written in a following way:

$${x}(n) = J_{\lambda}^{n-1} \cdot {x}(1) + \sum_{k=1}^{n-1} J_{\lambda}^{n-k-1} {y}(k)$$

Thus, for $1 \leq m \leq M$:

$$x_{m}(n) = \sum_{i=0}^{M-m} \left[ \binom{n-1}{i}\lambda^{n-1-i} \cdot x_{i+m}(1) +  \sum_{k=1}^{n-1} \binom{n-k-1}{i}\lambda^{n-k-1-i} \cdot y_{i+m}(k) \right]$$

We apply Lemma \ref{lemma_identity_binomial}:

$$\binom{n-k-1}{i} = \sum_{r=0}^{i} (-1)^{i-r} \binom{n-1}{r} \frac{(k)^{[i-r]}}{(i-r)!}$$

and obtain

$$x_{m}(n) = \sum_{i=0}^{M-m} \left[ \binom{n-1}{i}\lambda^{n-1-i} \cdot x_{i+m}(1) +  \sum_{k=1}^{n-1} \left(\sum_{r=0}^{i} (-1)^{i-r} \binom{n-1}{r} \frac{(k)^{[i-r]}}{(i-r)!}\right) \cdot \lambda^{n-k-1-i} \cdot y_{i+m}(k) \right] = $$

$$ = \sum_{i=0}^{M-m} \left[ \binom{n-1}{i}\lambda^{n-1-i} \cdot x_{i+m}(1) +  \sum_{r=0}^{i} \left\{ (-1)^{i-r} \binom{n-1}{r} \cdot \left(\sum_{k=1}^{n-1} \frac{(k)^{[i-r]}}{(i-r)!} \cdot \lambda^{n-k-1-i} \cdot y_{i+m}(k) \right) \right\} \right]$$

We change the order of summation:

$$\sum_{i=0}^{M-m} \sum_{r=0}^{i} \left\{ (-1)^{i-r} \binom{n-1}{r} \cdot \left(\sum_{k=1}^{n-1} \frac{(k)^{[i-r]}}{(i-r)!} \cdot \lambda^{n-k-1-i} \cdot y_{i+m}(k) \right) \right\} = $$

$$ = \sum_{r=0}^{M-m} \sum_{i=r}^{M-m} \left\{ (-1)^{i-r} \binom{n-1}{r} \cdot \left(\sum_{k=1}^{n-1} \frac{(k)^{[i-r]}}{(i-r)!} \cdot \lambda^{n-k-1-i} \cdot y_{i+m}(k) \right) \right\}$$

For the convenience we swap the indices $i$ and $r$:

$$\sum_{i=0}^{M-m} \sum_{r=i}^{M-m} \left\{ (-1)^{r-i} \binom{n-1}{i} \cdot \left(\sum_{k=1}^{n-1} \frac{(k)^{[r-i]}}{(r-i)!} \cdot \lambda^{n-k-1-r} \cdot y_{r+m}(k) \right) \right\}$$

Therefore,

$$x_m(n) = \sum_{i=0}^{M-m} \left[ \binom{n-1}{i}\lambda^{n-1-i} \cdot x_{i+m}(1) \right. + $$
    
$$ \left. + \sum_{r=i}^{M-m} \binom{n-1}{i} \left\{ (-1)^{r-i} \cdot \left(\sum_{k=1}^{n-1} \frac{(k)^{[r-i]}}{(r-i)!} \cdot \lambda^{n-k-1-r} \cdot y_{r+m}(k) \right) \right\} \right] = $$

$$ = \sum_{i=0}^{M-m} \binom{n-1}{i} \left[\lambda^{n-1-i} \cdot x_{i+m}(1) +  \sum_{r=i}^{M-m} \left\{ (-1)^{r-i} \cdot \sum_{k=1}^{n-1} \left( \frac{(k)^{[r-i]}}{(r-i)!} \cdot \lambda^{n-k-1-r} \cdot y_{r+m}(k) \right) \right\} \right] =$$

$$ = \sum_{i=0}^{M-m} \binom{n-1}{i} \left[\lambda^{n-1-i} \cdot x_{i+m}(1) +  \sum_{r=0}^{M-m-i} \left\{ (-1)^{r} \cdot \sum_{k=1}^{n-1} \left( \frac{(k)^{[r]}}{r!} \cdot \lambda^{n-k-1-r-i} \cdot y_{r+i+m}(k) \right) \right\} \right]$$

\begin{proposition}\label{solution_of_multidimensional_equation_jordan_cell}
    Solution of \eqref{multidimensional_equation_jordan_cell} can be written in a following way:
    
    $$x_m(n) = \sum_{i=0}^{M-m} \binom{n-1}{i} \lambda^{n-1-i} \left[x_{i+m}(1) +  \sum_{r=0}^{M-(i+m)} \left\{ (-1)^{r} \cdot \sum_{k=1}^{n-1} \left( \frac{(k)^{[r]}}{r!} \cdot \lambda^{-k-r} \cdot y_{r+i+m}(k) \right) \right\} \right],$$

    \noindent for $1 \leq m \leq M$.
\end{proposition}

\hspace{1pt}

\hspace{1pt}

The main theorem of this section follows.

\begin{theorem}
    Let $M \geq 2$, $\{\widetilde{y}_{m}(n) : n\geq 0\} \subset \mathbb{C}, \; 1 \leq m \leq M$ and the following conditions hold:

    \begin{itemize}
        \item For all $1 \leq m \leq M \, :$ $y_{m}(n) = \frac{\widetilde{y}_{m}(n)}{n^{m-1}}, \; n\geq 1$

        \item For each $m \in \{1,2,\ldots, M\}$ there exists some $\alpha_{m} \in \mathbb{C}$, such that the solution of the difference equation $z_{n+1} = \lambda z_{n} + \widetilde{y}_{m}(n), \, n \geq 1, \, z_{1} = \alpha_{m}$ is bounded. Equivalently, sequence of sums $\left\{\sum_{n=1}^{N} \widetilde{y}_{m}(n) \lambda^{-n}\right\}_{N\geq 1}$ is bounded.

        \hspace{1pt}

        Then the solution of the difference equation \eqref{multidimensional_equation_jordan_cell} is bounded if and only if $\forall m \in \{2, \ldots, M\} \, :$

        $$x_{m}(1) = -\sum_{r=0}^{M-m} \left\{ (-1)^{r} \cdot \sum_{k=1}^{\infty} \left( \frac{(k)^{[r]}}{r!} \cdot \lambda^{-k-r} \cdot y_{r+m}(k) \right) \right\}$$
    \end{itemize}

\end{theorem}

\noindent \textbf{Proof of The Main Theorem.}

Firstly, we would like to remark that each series 
$$\sum_{k=1}^{\infty} \left( \frac{(k)^{[r]}}{r!} \cdot \lambda^{-k-r} \cdot y_{r+m}(k) \right), \, 2 \leq m \leq M, \, 0 \leq r \leq M-m$$
\noindent is convergent by Dirichlet's test. Indeed, 

\begin{equation}\label{identity_1_main_theorem}
    \frac{(k)^{[r]}}{r!} \cdot \lambda^{-k-r} \cdot y_{r+m}(k) = \frac{\lambda^{-r}}{r!}\cdot \frac{\widetilde{y}_{r+m}(k) \lambda^{-k}}{\left(\frac{k^{r+m-1}}{(k)^{[r]}}\right)},
\end{equation}
sequence $\left\{\left(\frac{k^{r+m-1}}{(k)^{[r]}}\right)^{-1} \, : k \geq 1\right\}$ is eventually monotone and is convergent to 0 because $\frac{k^{r+m-1}}{(k)^{[r]}} \sim \frac{k^{r+m-1}}{k^{r}} = k^{m-1},\ k\to\infty$. Also, by the assumption, sequence of sums $\left\{\sum_{k=1}^{K} \widetilde{y}_{m}(k) \lambda^{-k}\right\}_{N\geq 1}$ is bounded.

\hspace{1pt}

    \textbf{Proof of sufficiency.}

\hspace{1pt}

We need the following Lemma:

\begin{lemma}
    Let $\{a_{n} : \, n \geq 1\} \subset \mathbb{R}$ be monotone, convergent to 0 sequence. And let $\{b_{n} : \, n \geq 1\} \subset \mathbb{C}$ be a sequence, for which there exists a constant $C > 0$, such that $\forall N \geq 1 \, : \left\vert \sum_{n=1}^{N} b_{n}\right\vert \leq C$. Then $\left\vert \sum_{k=n}^{\infty}a_{k}b_{k}\right\vert \leq C|a_{n}|$ for all $n \geq 1$.
\end{lemma}

This lemma follows from the Abel's summation by parts formula.

Now we use that
    $\forall m \in \{1,2,\ldots, m\} \; \exists C_{m} > 0 \; \forall N \geq 1: \left\vert \sum_{n=1}^{N} \widetilde{y}_{m}(n) \lambda^{-n} \right\vert \leq C_{m}.$

    \hspace{1pt}

    Therefore, by \eqref{identity_1_main_theorem}, for all $2 \leq m \leq M, \, 0 \leq r \leq M-m$ we have:

    $$\left\vert\sum_{k=N}^{\infty} \left( \frac{(k)^{[r]}}{r!} \cdot \lambda^{-k-r} \cdot y_{r+m}(k) \right)\right\vert \leq \frac{C_{m}}{r!} \cdot \frac{(N)^{[r]}}{N^{r+m-1}} \leq \frac{C_{m}}{r!} \cdot \frac{(N+r-1)^{r}}{N^{r+m-1}} = O\left(\frac{1}{N^{m-1}}\right),\ N\to\infty.$$

    Hence, since $\forall m \in \{2, \ldots, M\} \, :$

    $$x_{m}(1) = -\sum_{r=0}^{M-m} \left\{ (-1)^{r} \cdot \sum_{k=1}^{\infty} \left( \frac{(k)^{[r]}}{r!} \cdot \lambda^{-k-r} \cdot y_{r+m}(k) \right) \right\}$$

    We deduce that for all $1 \leq m \leq M$ and for all $0 \leq i \leq M-m$ we have:

    $$x_{i+m}(1) +  \sum_{r=0}^{M-m-i} \left\{ (-1)^{r} \cdot \sum_{k=1}^{n-1} \left( \frac{(k)^{[r]}}{r!} \cdot \lambda^{-k-r} \cdot y_{r+i+m}(k) \right) \right\} = O\left(\frac{1}{n^{i+m-1}}\right),\ n\to\infty.$$

    Since $\binom{n-1}{i} = O\left(n^{i}\right),\ n\to\infty$ using Proposition \ref{solution_of_multidimensional_equation_jordan_cell} we obtain that the solution of \eqref{multidimensional_equation_jordan_cell} is indeed bounded. 

    \hspace{1pt}

\textbf{Proof of necessity.}

    \hspace{1pt}

    Suppose that the solution $\{x(n) : \, n\geq 1\}$ of equation \eqref{multidimensional_equation_jordan_cell} is bounded. In particular, the sequence $\{x_{1}(n) : \, n \geq 1\}$ is bounded.

    \hspace{1pt}

    $x_{1}(n) = \sum_{i=0}^{M-1} A_{i}(n) \cdot \binom{n-1}{i}$, where

    $$A_{i}(n) := \lambda^{n-1-i} \left[x_{i+m}(1) +  \sum_{r=0}^{M-m-i} \left\{ (-1)^{r} \cdot \sum_{k=1}^{n-1} \left( \frac{(k)^{[r]}}{r!} \cdot \lambda^{-k-r} \cdot y_{r+i+m}(k) \right) \right\} \right], \; n \geq 1$$

    Then $A_{M-1}(n) \cdot \binom{n-1}{M-1} = O\left(n^{M-2}\right)$, so $A_{M-1}(n) = O\left(\frac{1}{n}\right) \Longrightarrow A_{M-1}(n) \to 0, \, n \to \infty$. Using the same argument from the proof of the \textit{sufficiency} we deduce that $A_{M-1} = O\left(\frac{1}{n^{(M-1) + 1 - 1}}\right) \Longrightarrow$ $A_{M-1}(n) \cdot \binom{n-1}{M-1} = O\left(1\right)$.

    Thus $A_{M-2}(n) \cdot \binom{n-1}{M-2} = O\left(n^{M-3}\right) \Longrightarrow A_{M-2}(n) \to 0, \, n \to \infty$ and analogously $A_{M-2}(n) \cdot \binom{n-1}{M-2} = O\left(1\right)$.

    Continuing this argument we show that $A_{m}(n) \to 0, \, n \to \infty$ for all $2 \leq m \leq M$.

\section*{Discussion and conclusions}

In this work we established novel insightful sufficient conditions for the existence of a bounded solution of the first order linear difference equation in one-dimensional case and further applied these conditions to the multidimensional case of this equation. The obtained results enrich the existing body of research in the field and complement previous studies.

\begin{tabular}{@{}l@{}}%
    A. Chaikovskyi \\
    \textsc{Taras Shevchenko National University of Kyiv, Ukraine}\\
    \textit{E-mail address}: \texttt{andriichaikovskyi@knu.ua}
\end{tabular}

\vspace{10pt}

\begin{tabular}{@{}l@{}}%
    O. Liubimov \\
    \textsc{Taras Shevchenko National University of Kyiv, Ukraine}\\
    \textit{E-mail address}: \texttt{liubimov\_oleksandr@knu.ua}
\end{tabular}

\end{document}